\documentclass[11pt, reqno]{amsart}

\usepackage{tikz}
\usepackage{tikz-cd}

\newtheorem{theorem}{Theorem}[subsection]
\newtheorem{lemma}[theorem]{Lemma}

\newtheorem{definition}[theorem]{Definition}

\newlength{\margins}
\usepackage[top=\margins,bottom=\margins,left=\margins,right=\margins]{geometry}
\usepackage{amsmath}
\numberwithin{equation}{section}

\DeclareMathOperator{\mfm}{\mathfrak{m}}
\DeclareMathOperator{\mfp}{\mathfrak{p}}

\DeclareMathOperator{\J}{\mathcal{J}}

\DeclareMathOperator{\C}{\mathbb{C}}

\DeclareMathOperator{\F}{\mathbb{F}}

\usepackage{amscd, mathrsfs, dsfont, mathdots}
\usepackage{amsmath,amssymb,amsthm,mathscinet}
\usepackage{marginnote}
\usepackage{enumerate}
\usepackage{easybmat,graphics}
\usepackage{etex, tikz, epic}
\usepackage{hyperref}
\hypersetup{colorlinks=true,citecolor=black,linkcolor=black}
\usepackage{etex}
\usepackage[all]{xy}
\usepackage[mathscr]{euscript}
\usepackage{pictexwd, dcpic}

\begin{document}

\title[${\rm SL}_*$ over local and ad\`ele rings]{${\rm SL}_*$ over local and ad\`ele rings: \\ $*$-euclideanity and \\ Bruhat generators}


\author[L. Guti\'errez]{Luis Guti\'errez Frez}

\author[L. Lomel\'i]{Luis Lomel\'i}

\author[J. Pantoja]{Jos\'e Pantoja${}^*$}


\date{February 2021}

\begin{abstract}
Let $(R,*)$ be a ring with involution and let $A = {\rm M}(n,R)$ be the matrix ring endowed with the $*$-transpose involution. We study ${\rm SL}_*(2,A)$ and the question of Bruhat generation over commutative and non-commutative local and ad\`elic rings $R$. An important tool is the property of a ring being $*$-Euclidean. In this regard, we introduce the notion of a $*$-local ring $R$, prove that $A$ is $*$-Euclidean and explore reduction modulo the Jacobson radical for such rings. Globally, we provide an affirmative answer to the question wether a commutative ad\`elic ring $R$ leads towards the ring $A$ being $*$-Euclidean; while the non-commutative ad\`elic quaternions are such that $A$ is $*$-Euclidean and ${\rm SL}_*$ is generated by its Bruhat elements if and only if the characteristic is $2$.
\end{abstract}

\maketitle

\let\thefootnote\relax\footnotetext{(*) The authors were partially supported by Fondecyt Grant 1171583.}

\section*{Introduction}

We study ${\rm SL}_*$ groups, introduced by Pantoja and Soto-Andrade \cite{PaSA2003}, over local and ad\`ele rings with involution, where we work over commutative and non-commutative rings with identity. At the base of the algebraic properties of ${\rm SL}_*$ groups are its Bruhat elements and the question of Bruhat generation. In \cite{PaSA2009}, the notion of a $*$-Euclidean ring is introduced, which provides a powerfult tool that allows us to infer that ${\rm SL}_*(2,A)$ is generated by its Bruhat elements when $A$ is such a ring with involution.

The non-commutative $*$-analogue of a special linear group, poses many interesting questions that are in accordance with the classical theory of Weil representations, Bruhat presentations and the Langlands program. In this article, we study for the first time the questions of $*$-Euclideanity and Bruhat generation for ad\`ele rings, including the quaternions in characteristic zero and charactersitic $p$. We develop general machinery along the way, where we are in many places guided by global questions posed in the ad\`elic setting. For example, we make a careful study of the local question of reduction modulo $\mathfrak{p}$, in fact, modulo the Jacobson radical. It is curious that we quickly arrive at the notion of a \emph{$*$-local ring}, for which we introduce a proper set of hypothesis in order to characterize a ring with a unique ideal that is stable under the involution and is maximal with this property.

Many examples arise by considering the ring of matrices $A = {\rm M}(n,R)$, when $R$ is a $*$-local ring or an ad\`ele ring with involution. The ring $A$ is itself endowed with the $*$-transpose involution; the involution of $A$ induced from that of $R$. We prove that $A$ is $*$-Euclidean for a list of examples of involutive rings $(R,*)$ provided in Theorem~\ref{thm:loc:glob:Bruhat}:
\begin{itemize}
   \item[(i)] $R$ a $*$-local ring.
   \item[(ii)] $R$ the ring of ad\`eles $\mathbb{A}_F$ of a global field $F$, with trivial involution.
   \item[(iii)] $R = \mathbb{A}_E$, where $E/F$ is a separable quadratic field extension with the involution dictated by the non-trivial Galois automorphism.
   \item[(iv)] $R = \mathbb{A}_{\mathbb{H}}$, where $\mathbb{H}$ is a quaternion algebra over a global field $F$, ${\rm char}(F) = 2$.
\end{itemize}
In all of these cases, ${\rm SL}_*(2,A)$ is generated by its Bruhat elements. It is noteworthy to observe in case (iv), of a quaternion algebra over a global field $F$, that the ring ${\rm M}(n,\mathbb{A}_{\mathbb{H}})$ is $*$-Euclidean and is generated by its Bruhat elements if and only if ${\rm char}(F) = 2$.

We now give a more detailed account of the contents of this article. The ${\rm SL}_*$ functor of Pantoja and Soto-Andrade for not necessarily commutative rings with involution and basic properties are reviewed in \S~\ref{SL*:prelim}; the main references being \cite{PaSA2003, PaSA2009}. In particular, in \S~\ref{*:Euc} we recall the notion of a $*$-Euclidean ring.

We arrive at the notion of a $*$-local ring $R$ in \S~\ref{q-local:inv}, where we quickly characterize these rings in the basic Lemma~\ref{one:two:qloc}. They are $1$- or $2$-local rings with involution. We refer to the former simply as a local ring with involution, where the Jacobson radical is the unique maximal prime ideal $\mathcal{J} = \mfp$. While the latter case has two distinct maximal prime ideals $\mfp$ and $\mfp^*$, its Jacobson radical being $\mathcal{J} = \mfp \cap \mfp^*$. We then proceed to study symmetric and invertible elements for the matrix ring $A = {\rm M}(n,R)$, endowed with the involution induced from the $*$-local ring $R$, and their behavior under reduction modulo the Jacobson radical $\mathcal{J}$ of $A$. More precisely, we treat both cases of $1$- and $2$-local rings with involution in a single swoop, study the projection map $\pi: A \xrightarrow{\ \pi \ } \overline{A} = A/\mathcal{J}$ and produce a section map $\sigma$ that is compatible with symmetry and invertibility in Lemma~\ref{proj:sec:q-loc:lem}.

We prove that the matrix ring $A = {\rm M}(n,R)$ over a $*$-local ring is $*$-Euclidean in Theorem~\ref{thm:q-local}; a result that Soto-Andrade proved when $R$ is a field \cite{SA1978}, and, for a division ring $R$, it is part of the ``Co-prime Lemma'' of Pantoja and Soto-Andrade \cite{PaSA2003}. For this, we show in Lemma~\ref{lem:local:A/J:A} that if $\overline{A} = A/\mathcal{J}$ is $*$-Euclidean, then so is $A$. And, in Lemma~\ref{2:div:rings} we establish $*$-Euclideanity for the new case of $R = D_1 \times D_2$, with $D_1$ and $D_2$ division rings related by an anti-automorphism $\varphi: D_1 \rightarrow D_2$ and the $\varphi$-flip transpose involution on $A = {\rm M}(n,R)$.

The ad\`eles over a commutative global field are studied in \S~\ref{adeles:def}, where $A={\rm M}(n,\mathbb{A}_F)$ is proved to be $*$-Euclidean in Theorem~\ref{thm:adeles:F}. A quadratic global field extension $E/F$, which we address in \S~\ref{quad:ext}, leads to $A={\rm M}(n,\mathbb{A}_E)$ being $*$-Euclidean for the involution induced from the non-trivial Galois element of ${\rm Gal}(E/F)$, Theorem~\ref{thm:adeles:quad}. Locally, we require Lemmas \ref{lem:local:A/J:A} and \ref{2:div:rings} for the semi-local ring $R = \mathcal{O}_F \times \mathcal{O}_F$ and a finite residue ring $R/\mathcal{J} = k_F \times k_F$, respectively, each of them with the flip involution.

The $*$-Euclidean property for the quaternions $\mathbb{H}$ over a local or a global field $F$, is studied in \S~\ref{quat}. The general basic theory is expounded in \cite{WeilNT} for central simple algebras in a manner that is independent of the characteristic, and is detailed in \cite{Vi} for $\mathbb{H}$. We here encounter fundamental differences depending on the characteristic being $2$ or not, see Theorem~\ref{thm:adeles:H}. Locally, we may have non-split or split quaternions at a place $v$ of $F$. The finitely many non-split cases $\mathbb{H}_v$ are division algebras over a local field $F_v$, while at the remaining infinitely many split places we have that $\mathbb{H}_v$ are matrix quaternions. The case of a finite place $v$ of $F$, leads to a non-Archimedean local field $F_v$ with ring of integers $\mathcal{O}_v$. And we also have quaternioninc rings $\mathcal{Q}_v$ at the non-Archimedean places, which can also be split or non-split.

Before continuing to inspect the global quaternions, we need a couple of local lemmas that we prove in the slightly more general settings of quaternion matrix rings and quaternion division algebras that exhibit the main difference of the characteristic being different than $2$ or not, Lemmas \ref{lem:n2:n*E} and \ref{H:char2}, respectively. The main $*$-Euclideanity result for the ad\`elic quaternions is Theorem~\ref{thm:adeles:H}, whose proof also requires two general local lemmas that are of independent interest for local split matrix rings of quaternions and for division algebras, namely Lemmas \ref{lem:char2:split} and \ref{lem:char2:split:localring}. In short, the ring of matrices $A = {\rm M}(n,\mathbb{A}_{\mathbb{H}})$ over the quaternion ad\`ele ring is $*$-Euclidean if and only if ${\rm char}(\mathbb{H}) = 2$.

In the final section, we begin with Dieudonn\'e and his non-conmutative determinant, and we extend a basic criterion for invertibility to general linear groups over $*$-local rings. The case of a commutative $2$-local ring with involution naturally arises while studying unitary groups at a split place of a quadratic extension of global fields, where there is a known connection to general linear groups. We review the connection between ${\rm SL}_*$ groups and the even unitary groups in \S~\ref{unitary:SL*}.

It is in \S~\ref{bruhat:gen} where we arrive at our main application of the $*$-Euclidean property, that is, ${\rm SL}_*$ groups and Bruhat generation for a list of examples over local and ad\`ele rings, which includes unitary groups and ${\rm SL}_*$ over quaternion rings. We recall the definition of the Bruhat elements, provide basic properties and state the main application in the form of Theorem~\ref{thm:loc:glob:Bruhat}. Having done much of the work for local and ad\`ele rings, a large part of the theorem follows from the results of \S~\ref{Eu:Br:Ad}, once we incorporate the result of Pantoja and Soto-Andrade that ${\rm SL}_*(2,A)$ has a set of Bruhat generators when $A$ is $*$-Euclidean \cite{PaSA2009}. However, special care must be taken in the case of the quaternions.

The quaternions over local and ad\`ele rings, their connection to ${\rm SL}_*$ groups and the question of Bruhat generation are scrutinized in \S~\ref{sl:quat}. In addition to completing the proof of Theorem~\ref{thm:loc:glob:Bruhat}, we record a couple of interesting quaternionic facts along the way. For example, a quaternion algebra $\mathbb{H}$ over a non-Archimedean local field $F$ leads towards special, Dieudonn\'e special and $*$-analogue special linear groups: ${\rm SL}(2,F)$, ${\rm SL}(2,\mathbb{H})$ and  ${\rm SL}_*(2,\mathbb{H})$. All three groups are distinct, and the first and the third are related by
\begin{equation*}
   {\rm SL}_*(2,\mathbb{H}) = D_{\mathbb{H}} \cdot {\rm SL}(2,F),
\end {equation*}
where $D_{\mathbb{H}}$ is the subgroup of ${\rm SL}_*(2,\mathbb{H})$ consisting of diagonal elements. Furthermore, it is curious to observe that each of ${\rm SL}(2,F)$ and ${\rm SL}(2,\mathbb{H})$, by Ihara's theoerem, has a further decomposition as an amalgamated product involving the corresponding congruence subgroup over the ring of integers and Iwahori subgroup; see \cite{Se}, for example.

We conclude by refining our theorem on the algebra $A = {\rm M}(n,\mathbb{A}_{\mathbb{H}})$, obtaining in this case that ${\rm SL}_*(2,A)$ has Bruhat generation precisely when ${\rm char}(\mathbb{H}) = 2$. This is done by incorporating a result of \cite{CrGuSz2020} on Bruhat generation for finite fields that generalizes to the setting of a local non-Archimedean split quaternionic algebra ring, together with the results of \S~\ref{quat}.

\section{Preliminaries on ${\rm SL}_*$}\label{SL*:prelim}

Let $R$ be a ring with $1$, endowed with an anti-automorphism $\alpha: R \rightarrow R$, $r \mapsto r^*$, with $\alpha$ either trivial or of order $2$. Notice that $\alpha = {\rm id}$ is only possible when $R$ is commutative. In general, $R$ is a not necessarily commutative ring with involution.

\subsection{${\rm SL}_*$ groups}

Let $\mathcal{A}_R$ be the category of involutive rings with identity $(A,*)$, such that $R \subset A$ and the involution of $A$ is compatible with that of the involutive ring $R$. Our main examples are obtained by taking the ring of $n\times n$ matrices $A={\rm M}(n,R)$ with entries in $R$. In this case, the involution is given by
\begin{equation}\label{nxn:invo}
   a^* = (a_{ji}^*), \text{ for } a = (a_{ij}) \in {\rm M}(n,R).
\end{equation}

Let $(A,*) \in \mathcal{A}_R$. We denote the center of $A$ by $Z_A$ and its group of invertible elements by $A^{\times}$. We let
\[ A^{\text{sym}} = \{ a \in A \mid a^* = a \}, \]
called the set of symmetric elements, and let
\[ {Z^{\times\,\text{sym}}_A} = Z_A \cap A^\times \cap A^{\text{sym}}, \]
be the set of central invertible symmetric elements of $A$.

If we let $\mathcal{G}$ denote the category of groups, then we have a functor
\begin{align*}
   \mathcal{SL}:\mathcal{A}_R &\rightarrow \mathcal{G} \\
   (A,*) &\mapsto {\rm SL}_*(2,A),
\end{align*}
originally defined by Pantoja and Soto-Andrade. We next recall several basic properties of ${\rm SL}_*$-groups, proved in \cite{PaSA2003}.

\subsection{General setting}
Let $(A,*) \in \mathcal{A}_R$. The involution on $A$  induces the involution on the ring of matrices ${\rm M}(2,A)$ given by \eqref{nxn:invo}. If we write  
\[ J = \left( \begin{array}{rr} \ 0 & 1 \\ -1& 0 \end{array} \right), \]
then ${\rm GL}_*(2,A)$ denotes the set of invertible matrices in $g\in {\rm M}(2,A)$ such that
 \[ g^{*}Jg = \delta(g) J, \] 
 for some $\delta(g) \in Z^{\times\,\text{sym}}_A$.
 
 The set ${\rm GL}_*(2,A)$ forms a group under matrix multiplication, in fact, $g \in {\rm GL}_*(2,A)$ implies $\delta(g^{-1}) = \delta(g)^{-1}$, $g^* \in {\rm GL}_*(2,A)$ and $\delta(g)=\delta(g^*)$. Furthermore, we have an epimorphism
\[ {\rm det}_* : {\rm GL}_*(2,A) \rightarrow Z_A^{\times\,{\rm sym}}, \]
given by 
\[
   {\rm det}_*(g) = ad^{*} - bc^{*}, \quad \quad    
   g= \left(\begin{matrix} a & b \\c & d \end{matrix}\right)\in 
   {\rm GL}_*(2, A).
\]

\begin{definition}
 The group  ${\rm SL}_*(2,A)$ is the kernel of the epimorphism $\det_*$.
\end{definition}

Making explicit the conditions on the matrices, we see that ${\rm SL}_*(2,A)$ is the group of matrices 
\begin{equation}\label{g:matrix}
   g=\left(\begin{matrix} a & b \\c & d \end{matrix}\right),
\end{equation}
with $a, b, c, d \in A$, satisfying the following:
\begin{equation}\label{sl*:relations}
   ad^* - bc^* = a^*d - c^*b = 1 \quad \text{and} \quad 
   ab^*, cd^*, a^*c, b^*d \in A^{\rm sym}.
\end{equation}

We observe that ${\rm SL}_*(2,A)$ may often be viewed as the isometry group, (and ${\rm GL}_*(2,A)$ as the similitude group with multiplier $\delta$) of the hermitian form
\[ h(x,y) = x^*Jy, \quad x, y\in A^2. \]
Here, $x \in A^2$ is identified with a column vector, and the involution $*$ of $A$ is extended to a map from column vectors to row vectors, by taking the involution $*$ of $A$ entry-wise and the transpose.

Another observation is that one retrieves the groups ${\rm GL}(2,A)$ and ${\rm SL}(2,A)$ when $A$ is commutative with trivial $*$. 

\subsection{$*$-Euclidean rings}\label{*:Euc}
Let $(A,*)\in \mathcal{A}_R$. We say that $A$ is \emph{$*$-Euclidean} if given $a, c \in A$ such that
\[ Aa + Ac = A, \quad a^*c \in A^{\rm sym}, \]
and setting $a = r_{-1}$, $c = r_0$, then there exist elements $s_0, \ldots, s_{n-1} \in A^{\text{sym}}$, $r_1, \ldots, r_{n-1} \in A$ and $r_n \in A^\times$ such that
\[ r_{i-1} = s_i r_i + r_{i+1}, \]
for $i = 0, \ldots, n-1$. When this is so, we say $n$ is the decomposition length of the pair $(a,c)$. The minimum  length valid for all possible pairs, if it exists, is the length of the $*$-Euclidean ring $A$.

As examples of  $*$-Euclidean rings $A$ we have: the ring of matrices ${\rm M}(n,F)$ with entries in a field \cite{SA1978}; and, ${\rm M}(n,D)$ with entries in a division ring \cite{PaSA2003}. In \S~3 we will prove that ${\rm M}(n,R)$  is $*$-Euclidean, where $R$ is a $1$-local, $2$-local or a commutative ad\`ele ring.

\section{Local rings with involution}\label{q-local:inv}

We say that a ring $R$ with identity is \emph{l-maximal} if every left maximal ideal is an ideal. There is the similar notion of \emph{r-maximal}, involving right maximal ideals.

Assume that $(R,*)$ is a ring with involution, then $R$ is $l$-maximal if and only if it is $r$-maximal. A stable ideal $\mathfrak{a}$ of $R$ is one such that $\mathfrak{a}^* = \mathfrak{a}$.

An $l$-maximal ring with involution $(R,*)$ is Dedekind finite, i.e., a ring where every element with a right inverse also has a left inverse; equivalently, a left invertible element is invertible. To see this, suppose $x \in R$ is right invertible, but not left invertible. Then there exists a left maximal ideal $\mfm$ of $R$ containing $x$. But $x$ being right invertible implies $\mfm R = R$, contradicting the assumption on $R$ that $\mfm$ is also a right ideal.

We observe that a finite direct product of local rings is $l$-maximal. However, the ring  of matrices $A={\rm M}(n,R)$ over a local ring has left maximal ideals but no maximal ideals for $n > 1$; its Jacobson radical is a stable ideal, the unique maximal stable ideal. In this case, we know that $A$ is semilocal \cite{La2001}, hence Dedekind finite.

\subsection{On the notion of a $*$-local ring}

Define a \emph{$*$-local ring} to be an $l$-maximal ring with involution $(R,*)$ having a unique stable maximal ideal.

\begin{lemma}\label{one:two:qloc}
A $*$-local ring $R$ has maximal spectrum
\[ {\rm MSpec}(R) = \left\{ \mfp, \mfp^*\right\} \subset {\rm Spec}(R), \]
and unique stable maximal ideal given by the Jacobson radical
\[ \J = \mfp \cap \mfp^*. \]
Furthermore, $D = R/\mfp$ is a division ring, and we have
\begin{equation*}
   x \in R^\times \ \Longleftrightarrow \ 
   \bar{x} \in \left(R/\mathcal{J}\right)^\times \ \Longleftrightarrow \
   x \notin \mfp, \ x^* \!\notin \mfp.
\end{equation*}
\end{lemma}
\begin{proof}
Let $\mathfrak{s}$ be the unique stable maximal ideal of $R$. Observe that $\mathfrak{s}$ is contained in a maximal ideal $\mathfrak{m}$ of $R$, and we have that $\mathfrak{m} \cap \mathfrak{m}^*$ is a stable ideal. By hypothesis $\mathfrak{s} \supset \mathfrak{m} \cap \mathfrak{m}^*$. If $x$ were an element of $\mathfrak{s} \setminus \mathfrak{m} \cap \mathfrak{m}^*$, then $x \notin \mathfrak{m}$ or $x^* \notin \mathfrak{m}$. If $x^* \notin \mathfrak{m}$, for example, we would then have
\[ R = Rx^* + \mathfrak{m} \subset 
   \mathfrak{s} + \mathfrak{m} \subset \mathfrak{m}, \]
a contradiction; and, similarly if $x \notin \mathfrak{m}$. Hence, we must have $\mathfrak{s} = \mathfrak{m} \cap \mathfrak{m}^*$.

Now, if $\mathfrak{a} \in {\rm MSpec}(R)$ were distinct from $\mathfrak{m}$, $\mathfrak{m}^*$, then so would $\mathfrak{a}^* \in {\rm MSpec}(R)$. But then
\begin{equation}\label{eq:2:4:max}
  \mathfrak{s} \supset 
  (\mathfrak{m} \cap \mathfrak{m}^*) + (\mathfrak{a} \cap \mathfrak{a}^*) = R,
\end{equation}
where the last equality can be seen by using the Chinese Remainder Theorem. However, equation~\eqref{eq:2:4:max} gives a contradiction. Therefore, we must have
\[ {\rm MSpec}(R) = \left\{ \mathfrak{m}, \mathfrak{m}^*\right\} \text{ and } \mathfrak{s} = \J, \]
the Jacobson radical.

An application of the Chinese Remainder Theorem, tells us that
\[ x + \J \in \left( R/\mathcal{J} \right)^\times \ \Longleftrightarrow \
   x + \mathfrak{m} \in \left( R/\mathfrak{m} \right)^\times\!\!, \
   x + \mathfrak{m}^* \in \left( R/\mathfrak{m}^* \right)^\times.
\]
Now, the Jacobson radical has the property that $1+y$ is invertible for every $y \in \J$. From here, we can infer that
\[ x \in R^\times \ \Longleftrightarrow \ 
   x + \J \in \left( R/\mathcal{J} \right)^\times \ \Longleftrightarrow \
   x \notin \mathfrak{m}, \ x \notin \mathfrak{m}^* . \]
Finally, $R$ is semilocal, hence Dedekind finite \cite{La2001}. Then a maximal ideal $\mfm \in {\rm MSpec}(R)$ is prime, i.e. $\mfm = \mfp \in {\rm Spec}(R)$, and the quotient $D = R/\mfp$ is a division ring.
\end{proof}

We thus have two possibilities for a $*$-local ring $R$,  depending if it has one or two maximal prime ideals. We refer to the former case as a $1$-\emph{local ring with involution}, or simply a \emph{local ring with involution}, since $R$ has a unique maximal ideal $\mfp = \mfp^*$. And call the latter a \emph{$2$-local ring with involution}, where $\mfp \neq \mfp^*$.

To give an example of a $2$-local ring with involution, take a local ring $R$ with maximal ideal $\mfp$, then we form the semilocal ring $S = R \oplus R$ and provide it with the flip involution $(r,r')^* = (r',r)$. Then $S$ is a $*$-local ring that is not a local ring.

\subsection{Reduction mod $\mfp$ for $*$-local rings}\label{q:loc:mod:p}

Let $R$ be a $*$-local ring with
\[ {\rm MSpec} = \{ \mfp, \mfp^* \}. \]
We consider the ring $S = R \oplus R$ with flip${}^*$ involution
\[ (x,y)^* = (y^*,x^*). \]
Then $R$ is isomorphic to the diagonal subring
\[ R_\Delta = \left\{ (z,z) \mid z \in R \right\} \subset S. \]
In this setting, the involution on $R$ is compatible with the flip${}^*$ involution
\begin{equation*}
   \begin{tikzcd} 
      R \arrow[r,"*"] \arrow[d,swap,"\wr"] & 
      R \arrow[d,"\wr"]\\
      R_\Delta \arrow[r,"{\rm flip}^*"] & R_\Delta
   \end{tikzcd}
\end{equation*}
We fix a maximal prime ideal $\mfp$ of $R$, and reduce mod $\J$. We write
\[ \overline{R} = R/\mathcal{J}, \quad \overline{S} = 
   R/\!\mfp\,\oplus\,R/\!\mfp^*, \]
where we have two projection maps
\[ x \in R \mapsto \bar{x} = x + \mfp \in R/\mfp \text{ and } y \in R \mapsto \tilde{y} = y + \mfp^* \in R/\mfp^*, \]
giving rise to a projection from $S$ to $\overline{S}$, $\pi : (x,y) \mapsto (\bar{x},\tilde{y})$. Let
\[ \overline{R}_{\mfp} =  \left\{ (\bar{z},\tilde{z}) \mid z \in R \right\} \subset \overline{S}, \]
so that we obtain a non-canonical projection map $\pi$.
\begin{equation*}
   \begin{tikzcd} 
      R_\Delta \arrow[r,"\sim"] \arrow[rrd,swap,"\pi"] & 
      R \arrow[r,"{\rm proj}"] & R/\mathcal{J} \arrow[d,"\wr"]\\
      & & \overline{R}_{\mfp}
   \end{tikzcd}
\end{equation*}
Here, $\overline{R}_{\mfp}$ is isomorphic to $R/\mfp$ when $R$ is a local ring, and is the degree $2$ separable algebra $R/\mfp \oplus R/\mfp^* \simeq R/\mathcal{J}$ when $R$ is a $2$-local ring, equipped with the involution
\[ (\bar{z}, \tilde{z}) \mapsto (\bar{z}^*, \tilde{z}^*). \]
In the latter case, notice that the isomorphism $R/\mathcal{J} \simeq \overline{R}_{\mfp}$ is obtained via the Chinese Remainder Theorem, where every $(\bar{x},\tilde{y})$ corresponds to a $(\bar{z},\tilde{z}) \in \overline{R}_{\mfp}$, $z \in R$.

We wish to construct a non-canonical section map $\sigma$ for $\pi$, in such a way that it is compatible with symmetry and invertibility. From Lemma~\ref{one:two:qloc}, we know that
\begin{equation*}
   z \in R^\times \simeq R_{\Delta}^\times \ \Longleftrightarrow \ 
   \pi(z) = (\bar{z}, \tilde{z}) \in \overline{R}_{\mfp}^\times.
\end{equation*}
We build a set consisting of pairs of representatives
\[ R_\sigma^\times = \{ (z_\sigma, z_\sigma) \in R_\Delta^\times \mid
   (\bar{z}_\sigma, \tilde{z}_\sigma) \in \overline{R}_{\mfp}^\times \}, \]
which satisfy
\[ z = z^* \ \Longleftrightarrow \ z_\sigma = z_\sigma^*. \]
We enlarge this set to obtain a section from $R_{\mfp}$ to $R_\Delta$, in such a way that it respects symmetry,
\[ R_\sigma = R_\sigma^\times \cup 
   \left\{ (z_\sigma,z_\sigma) \mid z_\sigma \in \mfp \cup \mfp^* \right\}. \]
Hence, by construction, we have
\[ z \in R^\times \simeq R_\Delta^\times \ \Longleftrightarrow \ z_\sigma \in R^\times. \]
and
\[ z \in R^{\, \rm sym} \simeq R_\Delta^{\, \rm sym} \ \Longleftrightarrow \ z_{\sigma} \in R^{\, \rm sym}. \]

We next extend this further to $A = {\rm M}(n,R)$, where $(A,*) \in \mathcal{A}_R$ for the involution given by \eqref{nxn:invo}, where in this context $\J$ denotes the Jacobson radical of $A$. We let
\[ \overline{A}_{\mfp} = {\rm M}(n,\overline{R}_{\mfp}) =
   \left\{ (\bar{a},\tilde{a}) \in {\rm M}(n,R/\!\mfp) \oplus {\rm M}(n,R/\!\mfp^*) 
   \mid a \in {\rm M}(n,R) \right\}, \]
which has involution
\[ (\bar{a},\tilde{a}) \mapsto (\bar{a}^*,\tilde{a}^*). \]
Hence, we also have a non-canonical projection $\pi$ in this setting.
\begin{equation}\label{eq:non:can:proj}
   \begin{tikzcd} 
      A \simeq {\rm M}(n,R_\Delta) \arrow[r,"{\rm proj}"] \arrow[r,bend right,start anchor={[xshift=-9ex]},end anchor={[xshift=5ex]},"\pi"'] & 
      \overline{A}_{\mfp} \simeq A/\mathcal{J}.
   \end{tikzcd}
\end{equation}

We continue by extending the section map on $R$, entry-wise for the elements of $A$. Setting
\[ A_\sigma = \left\{ a_\sigma=(a_{ij}) \in {\rm M}(n,R) \mid a_{ij} \in R_\sigma \right\}, \] 
allows for writing $A$ as a sum of two sets
\[ A = A_\sigma + \mathcal{J}, \]
with $A_\sigma \cap \mathcal{J} = \{ 0 \}$. Given $a \in A$, we obtain a unique decomposition
\begin{equation}\label{A:decomp}
   a = a_\sigma + a_{\j}, \quad a_\sigma \in A_\sigma, \ a_{\j} \in \mathcal{J}.
\end{equation}
Furthermore, symmetric elements are such that
\begin{equation}\label{decomp:sym}
   a = a^* \ \Longleftrightarrow \ a_\sigma= a_\sigma^* \text{ and } a_{\j} = a_{\j}^*.
\end{equation}
We summarize the basic properties of the above construction in the following.

\begin{lemma}\label{proj:sec:q-loc:lem}
Let $R$ be a $*$-local ring, and form the matrix ring $A = {\rm M}(n,R)$. For the projection map
\[ A \xrightarrow{\ \pi \ } \overline{A} = A/\mathcal{J}, \quad a \mapsto a + \J, \]
there exists a non-canonical section map
\[ \overline{A} \xrightarrow{\ \sigma \ } A, \quad 
   a + \J \mapsto a_\sigma. \]
The maps preserve the involutions on $A$ and $\overline{A}$, and the following properties hold:
\begin{itemize}
   \item[(i)] The projection $\pi(a)$ is symmetric if and only if the section $a_\sigma$ is symmetric.
   \item[(ii)] The units satify
   \[ a \in A^\times \ \Longleftrightarrow \ 
      \pi(a) \in \overline{A}^\times \ \Longleftrightarrow \
      a_\sigma,a_\sigma^* \in A^\times. \]
\end{itemize}
\end{lemma}
\begin{proof}
By definition, $\pi$ preserves the involutions. The existence of $\sigma$ is due to the decomposition \eqref{A:decomp}, where we have
\[ a^* = a_\sigma^* + a_{\j}^*. \]
Hence, the section preserves the involutions, namely,
\[ (a_\sigma)^* = (a^*)_\sigma. \]
Property (i) follows from \eqref{decomp:sym}.

For invertibility, first suppose $a_\sigma, a_\sigma^* \in A^\times$. The elements of the form
\begin{equation*}
   1 + x, \ \text{ with } x \in \mathcal{J},
\end{equation*}
are known to be invertible by Bass' Lemma~6.4 of \cite{Ba1964}. Hence
\[ a_\sigma^{-1}a = 1 + a_\sigma^{-1}a_{\j} \in A^\times, \]
and we conclude that $a \in A^\times$. Clearly $a \in A^\times \, \Longrightarrow \, \pi(a) \in \overline{A}^\times$. Now, suppose that $\pi(a) \in \overline{A}^\times$; and, for brevity write $\bar{a} = \pi(a)$, $\bar{b} = \pi(b)$. Then $\bar{a} \cdot \bar{b} = \bar{b}\bar{a} = \bar{1}$ for some $\bar{b} \in \overline{A}^\times$. By writing $a = a_\sigma + a_{\j}$, $b = b_\sigma + b_{\j}$, we obtain $a_\sigma b_\sigma = b_\sigma a_\sigma = 1$. Thus $a_\sigma \in A^\times$, and $a_\sigma^* \in A^\times$.
\end{proof}

\section{$*$-Euclideanity over local and ad\`ele rings}\label{Eu:Br:Ad}

We study involutive matrix rings over ad\`elic rings and their $*$-Euclidean property. This global situation, naturally poses questions for rings with involution in the local setting. Hence, we begin by establishing for $*$-local rings, a result that Soto-Andrade proved when $R$ is a field in Chapitre III, \S~1.2, Lemme~3 of \cite{SA1978}. And, for division rings $R$, it is part of the ``Co-prime Lemma'' of Pantoja and Soto-Andrade \cite{PaSA2003}. We then extend to ad\`elic rings over a global field $F$, we consider separable quadratic extensions of global fields $E/F$, in addition to quaternion algebras over $F$.

\subsection{Local setting} Throughout this subsection, we let $R$ be a $*$-local ring, and form the ring
\[ A = {\rm M}(n,R), \quad (A,*) \in \mathcal{A}_R, \]
with the involution induced from that of $R$. We let $\J$ denote the Jacobson radical of $A$, or sometimes the Jacobson radical of $R$, and it should be clear from context which one is being used.

\begin{theorem}\label{thm:q-local}
The ring of $n \times n$ matrices $A$ over a $*$-local ring is $*$-Euclidean.
\end{theorem}
We need a couple of results in order to prove this theorem, where the key point is reduction mod $\J$.

\begin{lemma}\label{lem:local:A/J:A}
If the ring $\overline{A} = A/\mathcal{J}$ is $*$-Euclidean, then so is $A$.
\end{lemma}
\begin{proof}
We have elements $a, c \in A$, which we can reduce mod $\mathcal{J}$, namely, we look at $\bar{a}, \bar{c} \in \overline{A} = {\rm M}(n,\overline{R})$. We have the hypothesis
\begin{align*}
   a^*c \in A^{\rm sym} \ 
   &\Longrightarrow \ \bar{a}^*\bar{c} \in \overline{A}^{\,\rm sym} \\
   Aa + Ac = A \ 
   &\Longrightarrow \ \overline{A}\bar{a} + \overline{A}\bar{c} = \overline{A}.
\end{align*}
Identify $A$ with ${\rm M}(n,R_\Delta)$ and $\overline{A} = A/\mathcal{J}$ with $\overline{A}_{\mfp}$, as in \eqref{eq:non:can:proj}, where we have a projection map $\pi$. Setting $\bar{a} = \bar{r}_{-1}$ and $\bar{c} = \bar{r}_0$, then, by assumption, there exist elements $\bar{s}_0, \ldots, \bar{s}_{n-1} \in A^{\text{sym}}$, $\bar{r}_1, \ldots, \bar{r}_{n-1} \in A$ and $\bar{r}_n \in A^\times$ such that
\[ \bar{r}_{i-1} = \bar{s}_i \bar{r}_i + \bar{r}_{i+1}, \]
for $i = 0, \ldots, n-1$.

The section map of Lemma~\ref{proj:sec:q-loc:lem} gives
\begin{equation*}
   a_\sigma, \ c_\sigma, \ s_{i,\sigma}, \ r_{i,\sigma} \in A_\sigma \subset A,
\end{equation*}
where $s_{i,\sigma} \in A^{\text{sym}}$, $r_{i,\sigma} \in A^{\text{sym}}$ and $r_{n,\sigma} \in A^{\times}$. We then have
\[ \overline{r}_{i-1,\sigma} = \overline{s_{i,\sigma} r_{i,\sigma} + r_{i+1,\sigma}} \text{ in } \overline{A} = A/\mathcal{J}. \]
Hence, there exists in every step an $x_i \in \mathcal{J}$ such that
\begin{equation*}
   r_{i-1,\sigma} + r_{i-1,\j} = 
   s_{i,\sigma}r_{i,\sigma} + (r_{i+1,\sigma}  + x_i).
\end{equation*}
The last term parenthesis, when $i+1 = n$, is the sum of a unit of $A$ plus an element of the radical of $A$, so it is a unit by Property~(ii) of Lemma~\ref{proj:sec:q-loc:lem}. Hence, the result follows.
\end{proof}

After reducing mod $\J$, there are two possibilities for $R$: it is either a $1$-local or a $2$-local ring with involution. The former gives $\overline{A} = {\rm M}(n,D)$, where $D$ is a division ring, a case proved by Pantoja and Soto-Andrade in Proposition~3.3 of \cite{PaSA2003}. The latter leads to a sum of two division rings after reduction mod $\J$, and we now prove the $*$-Euclidean property in this case.

\begin{lemma}\label{2:div:rings}
Let $D_1$ and $D_2$ be division rings, together with an anti-automorphism $\varphi: D_1 \rightarrow D_2$. Let $A_i = {\rm M}(n,D_i)$, $i = 1$, $2$, and extend $\varphi: A_1 \rightarrow A_2$, componentwise. Then, consider the ring $A = A_1 \oplus A_2$, with $\varphi$-flip transpose involution
   \[ (a_1,a_2)^* = (\varphi^{-1}(a_2)^t,\varphi(a_1)^t). \]
Let $a, c \in A$ be such that
\[ Aa + Ac = A, \quad a^* c = c^* a. \]
Then, there exist
\[ s \in A^{\rm sym} = \{ (x,\varphi(x)^t) \mid x \in A_1 \}, \quad r \in A^\times, \] 
satisfying
\begin{equation*}
   a = s c + r.
\end{equation*}
\end{lemma}
\begin{proof}
Let us observe that any $s \in A^{\rm sym}$ and $r \in A^\times$ satisfying $a = sc + r$, must also be such that
\begin{equation}\label{*:c:r}
   Ac + Ar = A.
\end{equation}
And, the symmetry relation $a^*c = c^*a$ leads to
\begin{equation}\label{*:rel:c:r}
   c^*r = r^*c.
\end{equation}
We write
\[ a = (a_1,a_2), \ c = (c_1,c_2), \ s = (s_{1}, s_{2}), \ r = (r_{1},r_{2}). \]

Now, we have by hypothesis
\[ A_1 a_1 + A_1 c_1 = A_1. \]
Since $A_1$ consists of $n \times n$ matrices with entries in a division ring $D_1$, this equation tells us that $a_1$ and $c_1$ must satisfy
\[ {\rm rank}(a_1) + {\rm rank}(c_1) \geq n. \]
Because of this, we can multiply $a_1$ and $c_1$ by products of elementary matrices, $e$ and $f$, in such a way that
\[ ea_1 + fc_1 = u, \]
where $u$ is a unipotent matrix, in particular, $u \in A_1^\times$. We can now go back and choose $s$, where we note that we only need to define the first component, since the symmetry requirement, $s \in A^{\rm sym}$, fixes the second component
\[ s_{1} = - e^{-1}f \ \Longrightarrow \ r = a - sc. \]
With such a choice, $r_{1} = u \in A_1^\times$. Next, we need the second component of $r$ to be a unit, and for this we look at equation \eqref{*:rel:c:r}, which gives
\[ \varphi(r_{1})^t c_{2} = \varphi(c_{1})^t r_{2}. \]
And, incorporating \eqref{*:c:r} leads to
\begin{align*}
   A_2 c_{2} + A_2 r_{2} = A_2 
   \ &\Longrightarrow \ A_2 \varphi(r_{1})^t c_{2} + A_2 r_{2} = A_2 \\
   \ &\Longrightarrow \ A_2 \varphi(c_{1})^t r_{2} + A_2 r_{2} = A_2 \\
   \ &\Longrightarrow \ A_2 r_{2} = A_2 \ \Longrightarrow \ r_{2} \in A_2^\times.
\end{align*}
\end{proof}

With the two lemmas at hand, the proof of Theorem~\ref{thm:q-local} is complete. We observe that, in both cases, the $*$-Euclidean length is $1$.

\subsection{The ad\`eles}\label{adeles:def}
Let $F$ be a global field, i.e., either a number field or a function field, and let $\mathcal{O}_F$ be its ring of integers and $\mathbb{A}_F$ its ring of ad\`eles \cite{WeilNT}. Given a place or valuation $v$ of $F$, we let $F_v$ denote its completion. If $v$ is non-Archimedean, we let $\mathcal{O}_v$ be the corresponding ring of integers. 

We wish to study ${\rm SL}_*$ groups over the ad\`eles, in fact, over $A={\rm M}(n,\mathbb{A}_F)$. For this, we fix some notation. At every place $v$ of $F$, we write $A_v$ for ${\rm M}(n,F_v)$. There are finitely many infinite places of $F$, where we write $v \mid \infty$, and have two possibilities: $A_v = {\rm M}(n,\mathbb{R})$ or $A_v = {\rm M}(n,\mathbb{C})$. The condition $v \mid \infty$ being empty in the case of function fields. On the other hand, finite places of $F$ are in correspondence with nonzero prime ideals of the ring of integers $\mathcal{O}_F$:
\[ \mfp \ \longleftrightarrow \ v. \]
At every finite place $v$ of $F$, we write $O_v$ for the maximal compact open subgroup ${\rm M}(n,\mathcal{O}_v)$. We recall that the matrix ring of ad\`eles is a restricted direct product
\[ A = {\rm M}(n,\mathbb{A}_F) = \prod{} ' (A_v:O_v). \]
Let $S$ be a finite set of places containing all $v \mid \infty$, and let
\begin{equation*}
   A^S = \prod_{v \in S} A_v \times \prod_{v \notin S} O_v \subset A.
\end{equation*}
Given an element $a \in A$, there is an $S$ as above such that $a \in A^S$, and we write
\begin{equation*}
   a = (a_v) = a_S \cdot a^S,
\end{equation*}
where $a_S$ has coordinates $a_v \in A_v$ at every place $v \in S$ and is $1$ for $v \notin S$; and, $a^S$ has $1$ for coordinate at every $v \in S$ and $a_v \in O_v$ at every $v \notin S$.

\begin{theorem}\label{thm:adeles:F}
The ring of matrices $A={\rm M}(n,\mathbb{A}_F)$ is $*$-Euclidean.
\end{theorem}
\begin{proof}
With notation as above, let $S$ be a finite set of places of $k$ such that
\begin{equation*}
   a, c \in A^S.
\end{equation*}
For $a_S, c_S \in \prod_{v \in S} A_v \hookrightarrow A$, we can go place by place where the local result, included in Theorem~\ref{thm:q-local}, is known for each of the fields $F_v$ by Soto-Andrade \cite{SA1978}, and there are only finitely many places $v \in S$. We thus have the $*$-Euclidean property with $s_S \in \prod_{v \in S} A_v^{\text{sym}}$ and $r_S \in \prod_{v \in S} A_v^\times$.

Now, at places $v \notin S$, Theorem~\ref{thm:q-local} gives the $*$-Euclidean property, where the decomposition length is $1$. Thus we can solve for the equation
\begin{equation*}
   a_v = s_v c_v + r_v,
\end{equation*}
with $s_v \in O_v^{\text{sym}}$ and $r_v \in O_v^\times$. In this way, we obtain $s^S \in \prod_{v \notin S} O_v \hookrightarrow A$ and $r^S \in \prod_{v \notin S} O_v^\times \hookrightarrow A$.

Finally, by setting
\begin{equation*}
   s = s_S \cdot s^S \in A \text{ and } r = r_S \cdot r^S \in A^\times, 
\end{equation*}
we obtain the desired $*$-Euclidean property for the ring of ad\`eles with decomposition length $1$.
\end{proof}

\subsection{Quadratic extensions.}\label{quad:ext} We now let $E/F$ be a separable quadratic field extension of the global field $F$, where we take the involution given by the non trivial Galois element $\alpha \in {\rm Gal}(E/F)$. For every finite absolute value $v$ of $F$, there are two possibilities, either $v$ remains inert with respect to $E$ or $v$ is split.
\begin{equation}\label{Valpo:fig}
\begin{tikzpicture}[baseline=(current bounding box.center)]
\begin{scope}
   \draw (-1,0) node {$\mathfrak{p}$};
   \draw[-,black] (-1,.25) -- (-1,1);
   \draw (-1,1.25) node {$\mathfrak{P}$};
   \draw (0,1.25) node {$\mathfrak{P}_1$};
   \draw (2,1.25) node {$\mathfrak{P}_2$};
   \draw[-,black] (1,0) -- (0,1);
   \draw (1,0) node [below] {$\mathfrak{p}$};
   \draw[-,black] (1,0) -- (2,1);
\end{scope}

   \draw[<->,thick,black] (3,.625) -- (4,.625);
   
\begin{scope}[shift={(6,0)}]
   \draw (-1,0) node {$v$};
   \draw[-,black] (-1,.25) -- (-1,1);
   \draw (-1,1.25) node {$w$};
   \draw (0,1.25) node {$w_1$};
   \draw (2,1.25) node {$w_2$};
   \draw[-,black] (1,0) -- (0,1);
   \draw (1,0) node [below] {$v$};
   \draw[-,black] (1,0) -- (2,1);
\end{scope}
\end{tikzpicture}
\end{equation}
In one case, $\mathfrak{P} = \mathfrak{p} \mathcal{O}_E$ is a prime ideal of $\mathcal{O}_E$ and we have corresponding prime ideals $\mfp_v$ of $\mathcal{O}_v = \mathcal{O}_{F_v}$ and $\mathfrak{P}_w$ of $\mathcal{O}_w = \mathcal{O}_{E_w}$, where we say there is one place $w$ above $v$, written $w \mid v$. In the other case $\mathfrak{P}_1 \mathfrak{P}_2 = \mfp \mathcal{O}_E$, and we obtain two places $w_1, w_2 \mid v$. Now, every infinite place $v$ of $F$, written $v \mid \infty$, leads towards two possibilities: one place $w \mid v$, when we must have $E_w/F_v = \mathbb{C}/\mathbb{R}$; or two places $w_1$, $w_2 \mid v$, when $E \otimes_F F_v \simeq \mathbb{R} \times \mathbb{R}$ or $E \otimes_F F_v \simeq \mathbb{C} \times \mathbb{C}$.

For every valuation $v$ of $F$, finite or infinite, we let
\[ E_v = E \otimes_F F_v \simeq \prod_{w \mid v} E_w. \]
We thus have that $E_v/F_v$ is a separable quadratic $F_v$-algebra with involution. When $E_v/F_v$ is a field extension, the involution for $E_v$ is given by the non-trivial Galois automorphism. When $E_v \simeq E_{w_1} \times E_{w_2}$, we have $E_{w_1} \simeq E_{w_2} \simeq F_v$, hence we fix $E_v$ to be the $F_v$-algebra $E_v = F_v \times F_v$ with the flip involution. By the $\check{\rm C}$ebotarev density theorem, each case of one or two places of $E$ above one for $F$ happens with density $1/2$. In the former case $v$ is inert, while in the latter $v$ is split.

We thus endow $\mathbb{A}_E$ with the involution $*$ obtained from the involution of the $F_v$-algebra $E_v$ at every place $v$ of $F$. The involution extends to the ad\`elic ring of matrices $A={\rm M}(n,\mathbb{A}_E)$, as in \eqref{nxn:invo}, giving $(A,*) \in \mathcal{A}_{\mathbb{A}_F}$.

Using $w$ to denote places of $E$ and $v$ for places of $F$, and writing $O_w = {\rm M}(n,\mathcal{O}_w)$ for finite $w$, the ring of matrices over the ad\`eles of $E$ is the restricted direct product as before
\[ A = {\rm M}(n,\mathbb{A}_E) = \prod_w{} ' (A_w:O_w). \]
However, in order to incorporate the involution, we group the places of $E$ according to the places of $F$, by setting for finite places
\[ R_v = \prod_{w \mid v} \mathcal{O}_w, \quad
   K_v = {\rm M}(n,R_v) \subset G_v= {\rm M}(n,E_v). \]
Then, we can rearrange the restricted direct product to obtain
\[ A = \prod_v{} ' (G_v:K_v). \]

\begin{theorem}\label{thm:adeles:quad}
Let $E/F$ be a quadratic extension of global fields. The ring of matrices $A={\rm M}(n,\mathbb{A}_E)$ is $*$-Euclidean of decomposition length 1.
\end{theorem}
\begin{proof}
Let $a$, $c \in A$ be such that $Aa+Ac=A$ and $a^*c=c^*a$. There exists a finite set of places $S$ of $F$, which includes all $v \mid \infty$, such that
\begin{equation*}
   a, c \in G^S = \prod_{v \in S} G_v \times \prod_{v \notin S} K_v \subset A.
\end{equation*}
We write
\begin{equation*}
   a = a_S \cdot a^S, \quad c = c_S \cdot c^S,
\end{equation*}
similar to what we did before, however, we are now grouping the places $w$ of $E$ that lie above each place $v$ of $F$. With these observations, we then follow the same argument used to prove Theorem~\ref{thm:adeles:F} in order to prove the result.
\end{proof}

\subsection{Quaternions.}\label{quat} We work over a local or global field, where the basic theory is expounded in \cite{WeilNT} for central simple algebras in a manner that is independent of the characteristic, and is detailed in \cite{Vi} for quaternion algebras. We do, however, encounter differences with regards to the $*$-Euclidean property depending on the characteristic being $2$ or not, see Theorem~\ref{thm:adeles:H}.

Given a local or a global field $F$, we let $(\mathbb{H},*) \in \mathcal{A}_F$ be a quaternion algebra over $F$. Up to isomorphism, there are two options for $\mathbb{H}$: either it is a division ring over $F$ or it is the ring of matrices ${\rm M}(2,F)$. Both options are possible except when $F=\C$, where the only quaternion algebra is $\mathbb{H}={\rm M}(2,\C)$. The case of a matrix algebra over $F$ in general is called the split quaternion algebra, where we take the involution to be
\begin{equation}\label{eq:H:split:inv}
   h^* = J h^t J^{-1}, \ h \in \mathbb{H} = {\rm M}(2,F).
\end{equation}
The division ring case is called the non-split quaternion algebra over $F$.

If $F$ is a non-Archimedean local field, we denote its ring of integers by $\mathcal{O}$, and we denote by $(\mathcal{Q},*) \in \mathcal{A}_{\mathcal{O}}$ the quaternionic ring of $\mathbb{H}$. The ring $\mathcal{Q}$ is ${\rm M}(2,\mathcal{O})$ with the involution given by \eqref{eq:H:split:inv} when $\mathbb{H}$ is split, and it is a non-commutative local ring with involution when $\mathbb{H}$ is a division ring. In fact, in the non-split case the ring $\mathcal{Q}$ is locally profinite much like the $\mathfrak{p}$-adic integers $\mathcal{O}$.

If $F$ is a global field, at every place $v$ of $F$ we let
\[ \mathbb{H}_v = \mathbb{H} \otimes_F F_v, \]
where indeed, each $\mathbb{H}_v$ is an $F_v$-quaternion algebra. If $v$ is a finite place of the global field $F$, we then write $F_v$ for the resulting non-Archimedean local field with ring of integers $\mathcal{O}_v$; furthermore, we denote by $\mathcal{Q}_v$ the quaternionic ring of $\mathbb{H}_v$. The ring $\mathbb{H}_v$ is split at almost every place, Chapter XI of \cite{WeilNT}.

Writing $\mathbb{A}_F$ for the ring of ad\`eles of a global field $F$, we have the ad\`elic quaternions
\[ \mathbb{A}_{\mathbb{H}} = \mathbb{H} \otimes_F \mathbb{A}_F. \]
They can equivalently be seen as a restricted direct product
\[ \mathbb{A}_{\mathbb{H}} = \prod_v{} ' (\mathbb{H}_v:\mathcal{Q}_v), \]
We have that $(\mathbb{A}_{\mathbb{H}},*) \in \mathcal{A}_{\mathbb{A}_F}$ with the involution obtained from that of $\mathbb{H}_v$ at every place, i.e.,
\[ a^* = (a_v^*), \quad a = (a_v) \in \mathbb{A}_{\mathbb{H}}. \]

Before continuing to inspect $*$-Euclideanity for these rings, we record two lemmas that arise in the local setting and already mark a difference between working in characteristic $2$ or not. For this, we extend the involution given by \eqref{eq:H:split:inv} to ${\rm M}(2,R)$ over any ring with identity $R$. However, $a^*a$ does not define a quaternion norm for general $R$ as in the commutative case.

\begin{lemma}\label{lem:n2:n*E}
Let $R$ be such that $2$ is a regular element and let $Q = {\rm M}(2,R)$, so that $(Q,*) \in \mathcal{A}_R$ with the involution obtained from \eqref{eq:H:split:inv}. Then $Q$ is not $*$-Euclidean.
\end{lemma}
\begin{proof}
Because of the hypothesis on $R$, we have
\[ Q^{\rm sym} = \left\{ \alpha I_2 \mid \alpha \in R \right\}. \]
The elements
\[ a = \left( \begin{array}{rr} 1&0 \\ 1&0 \end{array} \right), \  
   c = \left( \begin{array}{rr} 0&1 \\ 0&1 \end{array} \right) \in Q, \]
are such that
\[ Q a + Q c = Q \text{ and } 
   a^* c=c^* a = 0. \]
However, $a$ and $c$ cannot satisfy the $*$-Euclidean property.
\end{proof}

When the underlying ring is a division algebra $D$, we recall that the Dieudonn\'e determinant on ${\rm M}(2,D)$ is given by
\[ \det \left( \begin{array}{rr} \alpha & \beta \\ \gamma & \delta \end{array} \right) =
   \left\{ \begin{array}{cl} 	
      \alpha\delta & \text{if } \gamma = 0 \\
      \alpha\gamma\delta\gamma^{-1} - \gamma\beta & \text{if } \gamma \neq 0
   \end{array} \right. . \]

\begin{lemma}\label{H:char2}
Let $D$ be a division ring of characteristic $2$ and let $H = {\rm M}(2,D)$, so that $(H,*) \in \mathcal{A}_D$ with the involution obtained from \eqref{eq:H:split:inv}. If $a$, $c \in H$ are such that
\[ H a + H c = H, \]
then there exists an $s \in H^{\rm sym}$ such that
\[ r = a + sc \in H^\times. \]
\end{lemma}
\begin{proof}
If $a$ is invertible or zero, the Lemma is immediate. Hence, we assume that ${\rm rank}(a)=1$. For any unit $u_0 \in H^\times$ and $s \in H^{\rm sym}$, we observe that
\[ r = a + sc \in H^\times \ \Longleftrightarrow \ 
   ru_0 = au_0 + scu_0 \in Q^\times. \]
Thus, after taking $u_0$ to be a suitable product of elementary matrices, we can assume $a$ is of the form
\begin{equation}\label{eq:a:r1}
   a = \left( \begin{array}{rr} x & 0 \\ y & 0 \end{array} \right).
\end{equation}
Also, we must have ${\rm rank}(c) = 1$ or $2$.

Assume $c$ is invertible. When $D = F$ is a field, this case is easy because $c^*c \in F$ is the quaternion norm. For $D$ in general, when $c$ is invertible, one can take $s \in H^{\rm sym}$ to be of the form
\begin{equation}\label{eq:sc:triang}
   \left( \begin{array}{rr} \alpha & \beta \\ 0 & \alpha \end{array} \right), \
   \left( \begin{array}{rr} \alpha & 0 \\ \gamma & \alpha \end{array} \right) \text{ or }
   \left( \begin{array}{rr} 0 & \alpha \\ \alpha & 0 \end{array} \right),
\end{equation}
to get $sc$ in triangular or anti-triangular form with non-zero diagonal or anti-diagonal entries, respectively. Depending on the form of $a$, one can choose $\alpha$, $\beta$ and $\gamma$ so that
\[ r = a + sc \in H^\times. \]
Note that a careful consideration of the three cases for $a$ of the form $\eqref{eq:a:r1}$ is needed when $D = \F_2$.

Now, suppose ${\rm rank}(c)=1$, then one can choose an appropriate $s \in A^\times$ of one of the forms in \eqref{eq:sc:triang}, so that
\begin{equation*}
   sc = \left( \begin{array}{rr} x & y\\ 0 & 0 \end{array} \right) \text{ or }
   \left( \begin{array}{rr} 0 & 0 \\ x & y \end{array} \right).
\end{equation*}
We can take one or the other, depending on the form of $a$, to obtain
\[ a = sc + r, \ s \in H^{\rm sym}, \ r \in H^\times. \]
\end{proof}

We now study the ring $A$ of $n \times n$ matrices with entries in $\mathbb{A}_{\mathbb{H}}$. The proof of the next theorem gives another example of how a global question requires us to inspect what is happening locally in detail, where we prove a pair of lemmas along the way. Note that the involution on $A$, globally or locally, is obtained by combining the involution on the quaternion ring with the involution given by \eqref{nxn:invo}.

\begin{theorem}\label{thm:adeles:H}
Let $\mathbb{H}$ be a quaternion algebra over a global field $F$, and let $A = {\rm M}(n,\mathbb{A}_{\mathbb{H}})$ so that $(A,*) \in \mathcal{A}_{\mathbb{A}_{\mathbb{H}}}$ with the involution induced from that of $\mathbb{A}_{\mathbb{H}}$. Then
\begin{itemize}
   \item[(i)] $A$ is $*$-Euclidean of decomposition length $1$ when ${\rm char}(F)=2$.
   \item[(ii)] $A$ is not $*$-Euclidean when ${\rm char}(F) \neq 2$.
\end{itemize}
\end{theorem}

Let $a = (a_v)$, $c = (c_v) \in A$ be such that
\[ A a+ A c = A \text{ and } a^*c=c^*a. \]
At finite non-split places, $\mathcal{Q}_v$ is a local ring with involution, where we know the result locally holds. Now, the proof of Theorem~\ref{thm:adeles:H} follows the general outline that is present in the proofs of Theorems~\ref{thm:adeles:F} and \ref{thm:adeles:quad}. However, care needs to be taken when inspecting ${\rm M}(n,\mathcal{Q}_v)$ at finite split places.

Before we continue, we have two lemmas that we write in a slightly more general setting. Let $R$ be a ring with identity where $2 = 0$. In this setting, we have that
\[ J = \left( \begin{array}{rr} \ 0 & 1 \\ 1& 0 \end{array} \right) = J^{-1}, \]
and we take the block diagonal matrix
\[ B = \left( \begin{array}{ccc} \ J & \cdots & 0 \\ \vdots &  & \vdots \\ 0 & \cdots & J \end{array} \right) = B^{-1}. \]
Consider the matrix ring ${\rm M}(2n,R)$, with involution given by
\begin{equation}\label{M(n):inv}
   a^* = Ba^tB.
\end{equation}
We first address the case when $R=F$ is a field.

\begin{lemma}\label{lem:char2:split}
Let $F$ be a field of characteristic $2$ and let $\mathbb{H} = {\rm M}(2,F)$ be the split quaternions. Form the matrix ring $A = {\rm M}(n,\mathbb{H})$ with the involution induced from that of $\mathbb{H}$. Then $A$ is $*$-Euclidean of length $1$.
\end{lemma}
\begin{proof}
Consider $A$ as the ring of $2n\times 2n$ matrices with entries in $F$, whose involution is given by \eqref{M(n):inv}. We then have a non-degenerate bilinear pairing
\[ \left\langle \cdot\,, \cdot \right\rangle : F^{2n} \times F^{2n} \rightarrow F, \quad
   \left\langle v, w \right\rangle = v^t B w, \]
where $v$, $w$ are seen as column vectors of $F^{2n}$. The involution $a \mapsto a^*$, is precisely the adjoint for the pairing $\left\langle \cdot\,, \cdot \right\rangle$, to which we can apply Transversality and the coprime Lemma of Pantoja and Soto-Andrade, in particular, Proposition 3.3 of \cite{PaSA2003}. In this way, given $a$, $c \in A$, such that
\[ A a + A c = A, \quad a^* c = c^* a, \]
we obtain an $s \in A^{\rm sym}$ such that
\[ r = a + sc \in A^\times. \]
\end{proof}

We now require to extend the above to the case of a local ring.

\begin{lemma}\label{lem:char2:split:localring}
Let $R$ be a local ring of characteristic $2$ with maximal prime ideal $\mfp$. Let $Q = {\rm M}(2,R)$, so that $(Q,*) \in \mathcal{A}_R$ with the involution obtained from \eqref{eq:H:split:inv}. Form the matrix ring $A = {\rm M}(n,Q)$ with the involution induced from that of $Q$. Then $A$ is $*$-Euclidean of length $1$.
\end{lemma}
\begin{proof}
We can identify $A$ with ${\rm M}(2n,R)$, where the involution is given by \eqref{M(n):inv}. We next follow our general construction of \S~\ref{q:loc:mod:p} for reduction modulo the Jacobson radical $\J$, $A \xrightarrow{\ \pi \ } \overline{A} = A/\mathcal{J}$, which produces a section map $\overline{A} \xrightarrow{\ \sigma \ } A$. More specifically, equations \eqref{A:decomp} and \eqref{decomp:sym} lead towards Lemma~\ref{proj:sec:q-loc:lem} where $\pi$ and $\sigma$ preserve the involutions on $A$ and $\overline{A}$, and the units satisfy
\[ a \in A^\times \ \Longleftrightarrow \ 
   \pi(a) \in \overline{A}^\times \ \Longleftrightarrow \
   a_\sigma \in A^\times. 
\]
Lemma~\ref{lem:local:A/J:A}, extended to this setting, shows that if $\overline{A}$ is $*$-Euclidean the so is $A$. After reducing mod $\mfp$, we apply Lemma~\ref{lem:char2:split}.
\end{proof}

We now continue the proof of Theorem~\ref{thm:adeles:H}. At every split place of $F$, we have a non-Archimedean local field $F_v$ and split quaternions $\mathbb{H}_v={\rm M}(2,F_v)$. We also have a ring of integers $\mathcal{O}_v$ and ring of quaternions $\mathcal{Q}_v={\rm M}(2,\mathcal{O}_v)$.

When ${\rm char}(F)=2$, thanks to Lemma~\ref{lem:char2:split:localring}, we now have $*$-Euclideanity for the ring ${\rm M}(n,\mathcal{Q}_v)$. However, when ${\rm char}(F) \neq 2$, we use Lemma~\ref{lem:n2:n*E}, which gives an immediate negative answer to the case $n=1$ at every split non-Archimedean place; a result that directly extends to the case of $n>1$. This concludes the proof of Theorem~\ref{thm:adeles:H}.

\section{Examples and Bruhat generators}

We begin with the Dieudonn\'e determinant and $*$-local rings, observing a connection between ${\rm SL}_*$ groups and the general linear group. We also provide two non-trivial examples for the theory, namely, unitary groups arising as ${\rm SL}_*$ groups and ${\rm SL}_*$ over local and ad\`elic quaternions.

The $*$-Euclidean property is a strong one for a ring $A$, and leads to the important result of Pantoja and Soto-Andrade that ${\rm SL}_*(2,A)$ has a set of Bruhat generators \cite{PaSA2009}. We thus establish Bruhat generation over local and ad\`elic rings, to conclude with Theorem~\ref{thm:loc:glob:Bruhat}.

\subsection{${\rm GL}(n)$ and $*$-local rings}

Given a division ring $D$, the Dieudonn\'e determinant \cite{Ar1957}, gives a criterion for the invertibility of $n \times n$ matrices with entries in $D$:
\[ x \in {\rm GL}(n,D) \ \Longleftrightarrow \ x \in {\rm M}(n,D) \text{ and } \det(x) \neq 0. \]

Now, let $R$ be a $*$-local ring, with ${\rm MSpec}(R) = \{ \mfp, \mfp^*\}$. From \S~\ref{q:loc:mod:p}, we have a projection map from $R$ to $R/\!\mfp$ and, depending on $\mfp = \mfp^*$ or $\mfp \neq \mfp^*$, also a projection from $R$ to $\,R/\!\mfp^*$. This leads to the following map of matrix rings
\[ {\rm M}(n,R) \longrightarrow {\rm M}(n,R/\!\mfp) \oplus {\rm M}(n,R/\!\mfp^*), \]
where we write
\[ a \mapsto (\bar{a},\tilde{a}). \]

Over the $*$-local ring $R$, the general linear group consists of invertible $n \times n$ matrices ${\rm GL}(n,R)$, with its principal congruence subgroup
\[ K = \{  a \in M(n,R) \mid \pi(a) = I_n + \J \}, \]
where $\pi : R \rightarrow R/\mathcal{J}$ is the canonical projection map. We now observe that invertibility is compatible with the Dieudonn\'e determinant and reduction mod $\J$.

\begin{lemma}\label{lem:q:local:Dieu}
Let $R$ be a $*$-local ring. The following are equivalent for $a \in {\rm M}(n,R)$:
\begin{itemize}
   \item[(i)] $a \in {\rm GL}(n,R)$.
   \item[(ii)] There exists $b \in {\rm M}(n,R)$ such that $ba \in K$.
   \item[(iii)] $\det(\bar{a}) \neq 0$ and $\det(\tilde{a}) \neq 0$.
\end{itemize}
\end{lemma}
\begin{proof}
This follows from Lemma~\ref{proj:sec:q-loc:lem}; notice that ${\rm M}(n,R)$ is Dedekind finite.
\end{proof}

A basic and important example of a $2$-local ring with involution is $S = R_1 \times R_2$, obtained from local rings $R_i$ that are linked by an anti-automorphism $\varphi: R_1 \rightarrow R_2$, equipped with the $\varphi$-flip involution
\[ r^* = (\varphi^{-1}(r_2),\varphi(r_1)) \text{ for } r = (r_1,r_2) \in S. \]
Form the matrix ring
\[ A = {\rm M}(n,S) = A_1 \oplus A_2, \quad A_i = {\rm M}(n,R_i), \]
with the $\varphi$-flip transpose involution
\[ a^* = (\varphi^{-1}(a_2)^t,\varphi(a_1)^t) \text{ for } a = (a_1,a_2) \in A, \]
so that $(A,*) \in \mathcal{A}_R$. We recall that Proposition~5.1 of \cite{PaSA2003}, gives an isomorphism
\begin{equation}\label{eq:sl:gl:2loc}
   {\rm SL}_*(2,A) =
   \{ a = (a_1,J(\varphi(a_1)^{-1})^tJ^{-1}) \in A \mid a_1 \in 
   {\rm GL}(2,A_1) \}
   \simeq {\rm GL}(2,A_1).
\end{equation}

\subsection{Unitary groups}\label{unitary:SL*}
Classically, one works over a base field $F$, where we have the unitary group  ${\rm U}_{2n}(F)$ associated to a separable quadratic algebra $E$ over $F$ with involution $\alpha: E \rightarrow E$, $x \mapsto \bar{x}$, given by the non-trivial Galois element $\alpha \in {\rm Gal}(E/F)$ if $E/F$ is a field extension, and is obtained from the flip involution $(x,y) \mapsto (y,x)$ if $E \simeq F \times F$. Let
\[ \Phi_n = \begin{pmatrix} \ 0 & J_n \\ -J_n & 0 \end{pmatrix}, \]
where  $J_n$ is the $n \times n$ matrix $J_n = \left( \delta_{i,n-j+1} \right)$, where $\delta_{i,j}$ denotes Kronecker's delta function, and define the hermitian form $h_n$ on the $2n$-dimensional vector space $V$ of column vectors with entries in $E$, defined by
\[ h_n(x,y) = \bar{x}^t \, \Phi_n y, \quad x, y \in V. \]
Then ${\rm U}_{2n}(F)$ is the group of isometries of $h_n$.

There is a more general setting for defining unitary groups by  considering $(A,*) \in \mathcal{A}_R$ and let $\varepsilon=\pm 1$. Then,  a $\varepsilon$-hermitian form  over a free left $A$-module $V$  of finite rank is a bi-additive map $h:V\times V\to A$ , linear in the second variable and satisfies
\[
h(v,u)^*=\varepsilon h(u,v),\quad u,v\in V.
\]
The unitary group ${\rm U}(h)$ associated to $h$  consists then  of all $g\in{\rm GL}(V)$ such that 
\[ 
h(gu,gv) = h(u,v), \quad   u, v \in V. 
\]

One of our examples arises by taking $A = {\rm M}(n,\mathbb{A}_E)$, where we have the connection between unitary groups and ${\rm SL}_*$ groups,
\[ {\rm U}_{2n}(\mathbb{A}_F) = {\rm SL}_*(2,\mathbb{A}_E). \]
We  thus may observe that, at every split place $v$ of $F$, we further have a connection between ${\rm SL}_*$ groups and general linear groups. Indeed, in this case we can fix $E_v = F_v \times F_v$ and take the map $\varphi$ of \eqref{eq:sl:gl:2loc} to be the identity. More precisely, ${\rm SL}_*(2,E_v)$ is isomorphic to ${\rm GL}(n,F_v)$ via the projection map
\[
(g_0,J(g_0^{-1})^tJ^{-1})\mapsto g_0.
\]

\subsection{Bruhat Generators}\label{bruhat:gen}
We have the Bruhat elements of ${\rm SL}_*(2,A)$, which are the natural extension of those for ${\rm SL}_2(F)$ when $F$ is  a field. Namely, the matrices
\[
   h_a = \left( \begin{array}{l} a^* \,\ 0 \\ 0 \ \,\ a^{-1} \end{array} \!\!\!\right),
   (a\in A^{\times}),\quad
   u_b = \begin{pmatrix} 1&b\\0&1  \end{pmatrix},  (b \in A^{\rm sym}), \quad
   \text{ and } \quad
   \omega = \left( \begin{array}{rr} 0&1\\ -1&0 \end{array} \right)
\]
are the Bruhat elements for ${\rm SL}_*(2,A)$.

We observe a formal Bruhat relation, valid when one of the entries is a unit. In particular, if

\[ g = \left(\begin{matrix} a & b \\c & d \end{matrix}\right) \in {\rm SL}_*(2,A), \]
and $a \in A^\times$, then
\begin{equation}\label{Bruhat:ind:eq}
   g = w^{-1} h_{a}^{-1} \, u_{-a^*c} \, w \, u_{a^{-1}b}.
\end{equation}
Note that if $b$, $c$ or $d$ is in $A^\times$, then multipliying $g$ on the left or right by $w$, leads to a matrix in the previous situation. Notice that each one of the elements appearing in \eqref{Bruhat:ind:eq} is indeed inside ${\rm SL}_*(2,A)$, follows from the defining relations given by \eqref{sl*:relations}. For example,
\[ ab^* = ba^* \ \Longrightarrow \ a^{-1}b = b^*(a^*)^{-1} \in A^{\rm sym}. \]
Hence $u_{a^{-1}b} \in {\rm SL}_*(2,A)$.

However, an element of $ {\rm SL}_*(2,A)$ does not in general satisfy the property that one of its entries is a unit. When $A = {\rm M}(n,R)$, with $R$ a $*$-local ring, this only happens when $n = 1$, because one can reduce the $*$-determinant relation mod $\mathcal{J}$:
\[ \bar{a}\bar{d}^* - \bar{b}\bar{c}^* = \bar{1}, \]
forcing $\bar{a}\bar{d}^*$ or $\bar{b}\bar{c}^*$ to be a unit. Hence $ad^*$ or $bc^*$ is a unit, and two of the entries are thus units in this case. For $n >1$, we can arrange for a product
\[ u_{b_1} w \,u_{b_2} w^{-1} u_{b_3} \]
to have all of its entries non-invertible, after suitable choices for $b_1, b_2, b_3 \in A^{\rm sym}$.

Now, a very interesting problem is to determine when the matrices $h_t$, $u_s$, $w$ generate the group ${\rm SL}_*(2,A)$. In this sense, there exists an important relation between $*$-Euclidianity and Bruhat elements. More precisely, Pantoja and Soto-Andrade proved that if $A$ is $*$-Euclidean, then $h_a$, $u_b$ and $w$, with $a\in A^{\times}$  and $b\in A^{\rm sym}$, generate ${\rm SL}_*(2,A)$ \cite{PaSA2009}, \S~5 Proposition~4.

Thus, we can apply these observations to the four different settings of Theorems \ref{thm:q-local}, \ref{thm:adeles:F} and \ref{thm:adeles:quad}, from which the next result is a corollary; we elaborate on ${\rm SL}_*$ groups over the quaternions in the following subsection.

\begin{theorem}\label{thm:loc:glob:Bruhat}
Consider involutive rings $(R,*)$ in the following cases:
\begin{itemize}
   \item[(i)] $R$ is a $*$-local ring; or
   \item[(ii)] $R$ is the ring of ad\`eles $\mathbb{A}_F$ of a global field $F$, with trivial involution; or
   \item[(iii)] $R = \mathbb{A}_E$, where $E/F$ is a separable quadratic field extension with the involution dictated by the non-trivial Galois automorphism; or
   \item[(iv)] $R = \mathbb{A}_{\mathbb{H}}$, where $\mathbb{H}$ is a quaternion algebra over a global field $F$, ${\rm char}(F) = 2$.
\end{itemize}
Let $A = {\rm M}(n,R)$, so that $(A,*) \in \mathcal{A}_R$ with the involution induced from that of $R$. Then ${\rm SL}_*(2,A)$ is generated by the Bruhat elements $h_a$, $u_b$ and $w$.
\end{theorem}

When the underlying ring $R$ is local, it was proved in \cite{CrGuSz2020}, where many interesting properties are explored like the Weil representations of these groups. Part (i) of the previous theorem is an extension to the case of $*$-local rings and provides and alternate proof when $R$ is local.

The unitary groups, as in \S~\ref{unitary:SL*}, give an example of case (iii), where the proof of the theorem follows from the results of \S~\ref{Eu:Br:Ad}. And finally, the quaternions provide an interesting non-conmutative example of Theorem~\ref{thm:loc:glob:Bruhat} for case (i) in general, and case (iv) when ${\rm char}(F)=2$. We now inspect ${\rm SL}_*$ groups over the quaternions more closely, explore some interesting properties and complete the proof of the theorem.

\subsection{Quaternions}\label{sl:quat} Let $\mathbb{H}$ be a quaternion division algebra over a local or a global field $F$.  We have three related, yet distinct groups in this setting, namely
\[ {\rm SL}(2,F), \ {\rm SL}(2,\mathbb{H}) \text{ and } {\rm SL}_*(2,\mathbb{H}). \]
The first two are the special linear groups of $F$ and $\mathbb{H}$, respectively, where ${\rm SL}(2,\mathbb{H})$ is given by the kernel of the Dieudonn\'e determinant. The third is the ${\rm SL}_*$ group of Pantoja and Soto-Andrade over the division algebra $\mathbb{H}$ endowed with the quaternionic involution.

\medskip

\noindent{\bf Case of a $\mfp$-adic field}. Let $F$ be a non-Archimedean local field $F$, with ring of integers $\mathcal{O}$ and maximal ideal $\mfp$. We have a finite residue field $k_F = \mathcal{O}/\mfp$.

In this setting, $\mathbb{H}$ is the unique (up to isomorphism) $4$ dimensional central division algebra over $F$. As in \S~\ref{quat}, we let $\mathcal{Q}$ denote the ring of integers of $\mathbb{H}$, and $\mathfrak{q}$ its maximal ideal. We obtain a finite quotient field $k_{\mathbb{H}} = \mathcal{Q}/\mathfrak{q}$.

Consider the ring of matrices $A = {\rm M}(n,\mathcal{Q})$ with the involution induced from $(\mathcal{Q},*) \in \mathcal{A}_{\mathcal{O}}$. Then $A$ is $*$-Euclidean by Theorem~\ref{thm:q-local}. Hence, by \cite{PaSA2009} we know that ${\rm SL}_*(2,A)$ is generated by the Bruhat elements of \S~\ref{bruhat:gen} and Theorem~\ref{thm:loc:glob:Bruhat} is thus valid for the quaternionic ring $\mathcal{Q}$. In particular, ${\rm SL}_*(2,\mathbb{H})$ is generated by the Bruhat elements of \S~\ref{bruhat:gen}.

The special linear groups ${\rm SL}(2,F)$ and ${\rm SL}(2,\mathbb{H})$ are generated by its Bruhat elements, see for example \cite{Ar1957}. We further observe that we have group homomorphisms obtained from reduction modulo the corresponding maximal ideal, namely
\begin{align*}
   \varphi_F : {\rm SL}(2,\mathcal{O}) \rightarrow {\rm SL}(2,k_{F})
\end{align*}
and
\begin{align*}
   \varphi_{\mathbb{H}}: {\rm SL}(2,\mathcal{Q}) \rightarrow {\rm SL}(2,k_{\mathbb{H}}).
\end{align*}
The Borel subgroups, $B(k_F)$ of ${\rm SL}(2,k_F)$ and $B(k_{\mathbb{H}})$ of ${\rm SL}(2,k_{\mathbb{H}})$, consist of the respective upper triangular matrices. These lead towards the Iwahori subgroups, defined as the pre-images of the above maps:
\[ \mathcal{I}_F =  \varphi_F^{-1}(B(k_F)) \text{ and } \mathcal{I}_{\mathbb{H}} =  \varphi_{\mathbb{H}}^{-1}(B(k_{\mathbb{H}})). \]

Now, consider the subgroups of diagonal elements,
\[ D_F = \left\{ h_a \mid a \in F^\times \right\} \text{ and } 
   D_{\mathbb{H}} = \left\{ h_a \mid a \in \mathbb{H}^\times \right\}. \]
It then follows from the definitions that we have the following relation
\begin{equation}\label{eq:sl*:sl}
   {\rm SL}_*(2,\mathbb{H}) = D_{\mathbb{H}} \cdot {\rm SL}(2,F).
\end {equation}

It is interesting to note a result of Ihara, which gives an expression of the special linear group as an amalgamated product, valid over $F$,
\begin{equation}\label{amalgama:F}
   {\rm SL}(2,F) = {\rm SL}(2,\mathcal{O}) *_{\mathcal{I}_F} {\rm SL}(2,\mathcal{O}),
\end{equation}
and over $\mathbb{H}$,
\begin{equation}\label{amalgama:H}
   {\rm SL}(2,\mathbb{H}) = {\rm SL}(2,\mathcal{Q}) *_{\mathcal{I}_{\mathbb{H}}} {\rm SL}(2,\mathcal{Q}).
\end{equation}
Where a proof of equations \eqref{amalgama:F} and \eqref{amalgama:H} can be found in \cite{Se}. Also note that equation~\eqref{eq:sl*:sl} can also be written as an amalgamaded product,
\begin{equation}
   {\rm SL}_*(2,\mathbb{H}) = D_{\mathbb{H}} *_{D_F} {\rm SL}(2,F).
\end {equation}

\medskip

\noindent{\bf Case of a global field.} Let $F$ be either a number field or a global function field of characteristic $p$. At each place $v$ of $F$, we obtain a quaternion algebra $\mathbb{H}_v$ over the local field $F_v$; $\mathbb{H}_v$ may be split or non-split. At a non-Archimedean place $v$ of $F$, we have the ring of integers $\mathcal{O}_v$ of $F_v$, in addition to the quaternionic ring $\mathcal{Q}_v$ of $\mathbb{H}_v$. At these places, we let $Q_v = {\rm M}(n,\mathcal{Q}_v) \in \mathcal{A}_{\mathcal{Q}_v}$.

Let $A = {\rm M}(n,\mathbb{A}_{\mathbb{H}})$, and at every place $v$ of $F$ let $A_v = {\rm M}(n,\mathbb{H}_v)$. The matrix ring $A$ (resp. $A_v$ at each $v$) is equipped with the involution induced from that of $\mathbb{A}_{\mathbb{H}}$ (resp. $\mathbb{H}_v$). A straightforward inspection tells us that we have Bruhat generation for the group ${\rm SL}_*(2,A)$ if and only if we have Bruhat generation for ${\rm SL}_*(2,A_v)$ at every Archimedean place and for ${\rm SL}_*(2,Q_v)$ at every non-Archimedean place.

The quaternionic ring $\mathbb{H}$ over the global field $F$ forms a central division algebra over $F$. There is at least one place $v$ of $F$ where $\mathbb{H}_v$ is a division algebra, hence non-split; and, $\mathbb{H}_v$ can be non-split at only finitely many $v$ \cite{WeilNT}.

Fix a split place $v$ of $F$, where we have that $\mathbb{H}_v$ and $\mathcal{Q}_v$ are matrix algebras over $F_v$ and $\mathcal{O}_v$, respectively. The involutions on $\mathbb{H}_v$ and $\mathcal{Q}_v$ come from \eqref{eq:H:split:inv}. These, in turn, induce the involutions on $A_v$ and $Q_v$, respectively. When ${\rm char}(F)=2$, in Lemma~\ref{lem:char2:split:localring} we showed that $Q_v$ is $*$-Euclidean of length $1$. And, Lemma~\ref{lem:char2:split} says that $A_v$ is also $*$-Euclidean of length $1$. By \cite{PaSA2009}, \S~5 Proposition~4, ${\rm SL}_*(2,Q_v)$ is in this case generated by its Bruhat elements at every non-Archimedean place and so is ${\rm SL}_*(2,A_v)$. 

However, continuing with the case of a split place $v$, if ${\rm char}(F)\neq 2$, then it follows directly from Lemma~\ref{lem:n2:n*E} that $Q_v$ is not $*$-Euclidean. Furthermore, from \cite{CrGuSz2020} Example 13.3 for a finite field, we 
deduce  that ${\rm SL}_*(2,Q_v)$ is not generated by its Bruhat elements.

At a non-split place, we have either a division algebra $\mathbb{H}_v$ or a local ring with involution $\mathcal{Q}_v$. Hence, we know that ${\rm SL}_*(2,A_v)$ and ${\rm SL}_*(2,Q_v)$ are generated by its Bruhat elements at every non-split place.

Finally, we observe that only in the case of a global function field $F$ with ${\rm char}(F)=2$, do we have that $A_v$ is generated by its Bruhat elements at every place and so is $Q_v$ at all non-Archimedean places. This concludes the proof of Theorem~\ref{thm:loc:glob:Bruhat} in the remaining case. In particular, we have the following:
\begin{quote}
\emph{Let $\mathbb{H}$ be a quaternion algebra over a global field $F$ and let $\mathbb{A}_{\mathbb{H}}$ be its ring of ad\`eles. Let $A = {\rm M}(n,\mathbb{A}_{\mathbb{H}})$ be endowed with the involution induced from that of $\mathbb{A}_{\mathbb{H}}$. Then ${\rm SL}_*(2,A)$ is generated by its Bruhat elements if and only if ${\rm char}(F) = 2$.}
\end{quote}

\bigskip\bigskip

\noindent{\sc \Small Luis Guti\'errez Frez, Instituto de Ciencias F\'isicas y Matem\'aticas, Universidad Austral de Chile, Campus Isla Teja SN, Valdivia, Chile.}

\emph{\Small E-mail address: }\texttt{\Small luis.gutierrez@uach.cl}

\medskip

\noindent{\sc \Small Luis Lomel\'i, Instituto de Matem\'aticas, Pontificia Universidad Cat\'olica de Valpara\'iso, Blanco Viel 596, Cerro Bar\'on, Valpara\'iso, Chile.}

\emph{\Small E-mail address: }\texttt{\Small luis.lomeli@pucv.cl}

\medskip

\noindent{\sc \Small Jos\'e Pantoja, Instituto de Matem\'aticas, Pontificia Universidad Cat\'olica de Valpara\'iso, Blanco Viel 596, Cerro Bar\'on, Valpara\'iso, Chile.}

\emph{\Small E-mail address: }\texttt{\Small jose.pantoja@pucv.cl}

\end{document}